\documentclass{amsart}
\usepackage{amssymb}
\usepackage[foot]{amsaddr}
\usepackage{stmaryrd} 
\usepackage{amsmath} 
\usepackage{amscd}
\usepackage[toc]{appendix}
\usepackage{amsthm}
\usepackage{amsbsy}
\usepackage{mathtools}
\usepackage{tablefootnote}
\usepackage{multirow}
\usepackage[normalem]{ulem}
\useunder{\uline}{\ul}{}
\usepackage[margin=1.2 5in]{geometry}
\usepackage{commath, longtable}
\usepackage{comment, enumerate}
\usepackage{geometry}
\usepackage[matrix,arrow]{xy}
\usepackage{hyperref}
\usepackage{mathrsfs}
\usepackage{ulem}
\usepackage{color}
\usepackage{float}
\usepackage{mathtools,caption}
\usepackage[table,dvipsnames]{xcolor}
\usepackage{tikz-cd}
\usepackage{longtable}
\usepackage[utf8]{inputenc}
\usepackage[OT2,T1]{fontenc}
\usepackage[normalem]{ulem}
\usepackage[customcolors]{hf-tikz}
\usepackage[table]{xcolor}
\usepackage{enumitem}
\newlist{primenumerate}{enumerate}{1}
\setlist[primenumerate,1]{label={\arabic*$'$}}
\hypersetup{
 colorlinks=true,
 linkcolor=DarkOrchid,
 filecolor=blue,
 citecolor=olive,
 urlcolor=orange,
 pdftitle={HK Watkin's project},
 }
\usepackage{booktabs}

\DeclareSymbolFont{cyrletters}{OT2}{wncyr}{m}{n}
\DeclareMathSymbol{\Sha}{\mathalpha}{cyrletters}{"58}

\newtheorem{theorem}{Theorem}[section]
\newtheorem{Lemma}[theorem]{Lemma}

\newtheorem{proposition}[theorem]{Proposition}
\newtheorem{corollary}[theorem]{Corollary}
\newtheorem{definition}[theorem]{Definition}

\numberwithin{equation}{section}

\theoremstyle{remark}
\newtheorem{remark}[theorem]{Remark}
\newtheorem{example}[theorem]{Example}

\newcommand{\EC}{\mathsf{E}}
\newcommand{\rk}{\operatorname{rk}}

\newcommand{\R}{\mathbb{R}}
\newcommand{\Zp}{\mathbb{Z}_p}

\newcommand{\Z}{\mathbb{Z}}
\newcommand{\Q}{\mathbb{Q}}

\newcommand{\Sel}{\mathrm{Sel}}
\newcommand{\ord}{\mathrm{ord}}

\makeatletter
\theoremstyle{plain} 
\newtheorem*{intr@thm}{\intr@thmname}

\newtheorem*{c@njecture}{\conjn@name}
\newcommand{\myl@bel}[2]{
 \protected@write \@auxout {}{\string \newlabel {#1}{{#2}{\thepage}{#2}{#1}{}} }
 \hypertarget{#1}{}
 } 

\makeatother

\makeatletter
\newcommand{\mylabel}[2]{#2\def\@currentlabel{#2}\label{#1}}
\makeatother

\title[Elliptic curves of conductor $2^m p$, quadratic twists, and Watkins' conjecture]{Elliptic curves of conductor $2^m p$, quadratic twists, and Watkins' conjecture}

\author[J.~Hatley]{Jeffrey Hatley}
\address[Hatley]{
Department of Mathematics\\
Union College\\
Bailey Hall 202\\
Schenectady, NY 12308 USA\\ 
hatleyj@union.edu}

\author[D.~Kundu]{Debanjana Kundu}
\address[Kundu]{Department of Mathematics and Statistics\\ University of Regina \\ 3737 Wascana Pkwy\\ Regina, SK S4S 0A2 Canada\\ debanjana.kundu@uregina.ca}

\keywords{elliptic curves, rank, Watkins conjecture}
\subjclass[2020]{Primary 11G05}

\begin{document}

\begin{abstract}
Let $\EC/\mathbb{Q}$ be an elliptic curve.
By the modularity theorem, it admits a surjection from a modular curve $X_0(N) \to \EC$, and the minimal degree among such maps is called the \textit{modular degree} of $\EC$.
By the Mordell--Weil Theorem, $\EC(\Q)\simeq \mathbb{Z}^r \oplus T$ for some nonnegative integer $r$ and some finite group $T$.
Watkins' Conjecture predicts that $2^r$ divides the modular degree, thus suggesting an intriguing link between these geometrically- and algebraically-defined invariants.
We offer some new cases of Watkins' Conjecture, specifically for elliptic curves with additive reduction at $2$, good reduction outside of at most two odd primes, and a rational point of order two.

\medspace

\noindent\textsc{Résumé.}
Soit $\EC/\Q$ une courbe elliptique. D'après le théorème de modularité, il existe une surjection d'une courbe modulaire $X_0(N) \rightarrow \EC$, et le degré minimal parmi de telles applications est appelé le \textit{degré modulaire} de $\EC$. D'après le théorème de Mordell--Weil, on a $\EC(\Q) \simeq \Z^r \oplus T$ pour un entier $r \geq 0$ et un groupe fini $T$. La conjecture de Watkins prédit que $2^r$ divise le degré modulaire, ce qui suggère un lien intrigant entre ces invariants, qui sont définis géométriquement et algébriquement, respectivement. Nous présentons de nouveaux cas de la conjecture de Watkins, en particulier pour les courbes elliptiques ayant réduction additive en $2$, bonne réduction en dehors d'au plus deux nombres premiers impairs, et un point rationnel d'ordre deux.

\end{abstract}

\maketitle

\section{Introduction} 
The celebrated Modularity Theorem \cite{Wil95, TW95, breuil2001modularity} assures us that, given any elliptic curve $\EC/\Q$, there is a modular parameterization 
\begin{equation}\label{eq:modular-param}
X_0(N) \xrightarrow{\phi_{\EC}} \EC
\end{equation}
where $X_0(N)$ is the modular curve of level $N=\mathrm{conductor}(\EC)$.
Among all such maps, the least degree is called the \textit{modular degree} of $\EC$, which we denote by $m_\EC$.
The arithmetic significance of $m_\EC$ has been the subject of much research; for instance, its prime divisors are closely related to the congruence primes of the modular form $f_{\EC}$ associated to $\EC$ \cite{agashe2012modular} .

A famous conjecture of M.~Watkins \cite{watkins2002computing} predicts that $\ord_2(m_\EC)$ is bounded below by the Mordell--Weil rank of $\EC(\Q)$.
Much progress has been made on this problem, especially in the case when $m_\EC$ is odd (in which case Watkins' conjecture implies that $\EC(\Q)$ is finite); see for example \cite{CE09,KKwatkins,KKwatkins-corrigendum,yazdani}.

When $m_\EC$ is not assumed to be odd, much progress has been made in proving Watkins' Conjecture in quadratic twist families of elliptic curves.
For instance, in \cite{E-LP}, the authors show that if $\EC(\Q)[2]\neq 0$, then Watkins' conjecture holds for quadratic twists of $\EC$ by square-free integers with sufficiently many prime divisors.
In \cite{CaroPasten} the authors establish Watkins' conjecture for many semi-stable elliptic curves with $\EC(\Q)[2]\neq 0$ under some additional restrictions on the primes of split and non-split multiplicative reduction.
Most recently, J.~Caro \cite{caro2022watkins} proves that if $\EC$ is an elliptic curve with prime-power conductor and $\EC(\Q)[2]\neq0$, then Watkins' Conjecture holds for any quadratic twist of $\EC$.

Note that Watkins' conjecture predicts that $2 \mid m_\EC$ whenever $\EC(\Q)$ is infinite.
As explained in \cite{caro2022watkins}, the only missing case for this weaker conjecture is the case when $N$ is divisible by at most two odd primes, $\EC$ has additive reduction at $2$, and $\EC(\Q)[2]$ is nontrivial; see Appendix~\ref{appendix:clarify-caro} for a discussion of this claim.
The results of \cite{caro2022watkins} cover the case when $\EC$ has additive reduction at $2$ and at the odd primes, but it does not allow for any odd primes of multiplicative reduction.

The goal of this paper is to complement the existing work in this area by establishing Watkins' conjecture for many elliptic curves with additive reduction at $2$, bad reduction at no more than 2 odd primes, and $\EC(\Q)[2] \neq 0$.
To avoid overlapping with \cite{caro2022watkins}, we study elliptic curves with \textit{multiplicative} reduction at one odd prime.
We note that the curves we consider often have \textit{split} multiplicative reduction at an odd prime, so they are not covered by the main result of \cite{CaroPasten}.

We now summarize our results.

\begin{itemize}
\item We begin with the complete classification, due to W.~Ivorra, of elliptic curves of conductor $2^m p$ for $m \geq 2$ and $p$ an odd prime.
In Theorem~\ref{thm:rank-bound-for-ivorra-curves}, we prove that many curves with conductor $2^m p$ have rank at most 1, including all the curves with $2 \leq m \leq 5$.
In Corollary~\ref{cor:watkins-holds-for-ivorra-type}, we show that the rank bound implies Watkins' Conjecture for these curves, assuming either the Birch and Swinnerton-Dyer Conjecture or the finiteness of Shafarevich--Tate groups.
\item We then study quadratic twists of Ivorra curves, which have conductor $2^m p q^2$, and deduce Watkins' conjecture for many of these twists by bounding their ranks and using properties of the Petersson norm; see Section~\ref{sec:petterson-norm} and in particular Theorem~\ref{watkin's for twist using petterson norm}.
\item An appendix gives the full details of the proof of Theorem~\ref{thm:rank-bound-for-ivorra-curves}.
In the main body of the paper, we prove just one case in order to streamline the paper for the reader's convenience.
\end{itemize}

It is likely possible to extend these methods to study the remaining elliptic curves of conductor $2^m p^a q^b$, for instance those with conductor $2^m p q$, using the results of \cite{Mulholland_2006}.

\section*{Acknowledgements} The authors thank Daniel Kriz for answering their questions, and the referee for helpful suggestions which improved the paper. We also thank Antonio Lei for proofreading our French abstract. 
Partial support for this research was provided to the first-named author by an AMS--Simons Research Enhancement Grant for Primarily Undergraduate Institution Faculty.
DK was supported by an AMS--Simons early career travel grant.

\section{Ivorra's classification}
By \cite{LRS-conductor-bound}, we know that there are no elliptic curves ${\EC}/\Q$ with conductor divisible by $2^9$.
Curves with rational $2$-torsion and conductor $2^mp$, with $1 \leq m \leq 8$ an integer and $p \geq 29$ a prime, were classified by Ivorra \cite{ivorra2004courbes}.
Such curves come in $2$-isogenous pairs, and their Weierstrass forms all fall into certain families.

We summarize his results in the following theorem, omitting a few cases, as we now explain.
First, we omit the case $m=1$ since we are concerned with elliptic curves with \textit{additive reduction} at $2$.
Next, we omit the single pair of curves of conductor $8 \cdot 31$ described in \cite[Th\'{e}or\`{e}me 3 (4)]{ivorra2004courbes}, since it is easy to verify anything one wishes about a particular elliptic curve by e.g. looking them up on \cite{lmfdb}.
Finally, we omit the curves belonging to an isogeny class of size $4$ (see parts (3) and (6) of Th\'{e}or\`{e}me 4 and parts (4) and (6) of Th\'{e}or\`{e}me 6  of \cite{ivorra2004courbes}); this is in order to give a more uniform treatment to the curves we do consider.

We sort the remaining curves into convenient families, described below, which are amenable to explict $2$-descent arguments.
Each family is introduced in the order of first appearance in the statements \cite[Th\'{e}or\`{e}mes 3--8]{ivorra2004courbes}.
The exact correspondence is given in Table~\ref{table:dictionary} in Appendix~\ref{appendix-dictionary}.

We use the convention of \emph{op.~cit}. that, if an integer $n$ is a perfect square, then 
\[
\begin{cases}
\sqrt{n} \equiv 1 \pmod 4 & \text{if $n$ is odd}\\ 
\sqrt{n} \geq 0 & \text{if $n$ is even}.
\end{cases}
\]
Write $\omega(n)$ to denote the number of prime divisors of $n$.
We record some values of $\omega$ in the following theorem because it will be useful for our purposes.
Let $f \colon \mathbb{N} \to \mathbb{N}$ denote the function defined by 
\[
f(n) = \begin{cases}
18 + 2 \log_2(n) & \text{if}\ n <2^{96}\\
435 + 10 \log_2(n) & \text{if}\ n \geq 2^{96}.
\end{cases}
\]

\begin{theorem}
\label{thm:general-ivorra-type}
Let $p \geq 29$ be a prime and $k \geq 2$ an integer.
Suppose $\beta \in \Z$ is a perfect square, and let $\alpha = \sqrt{\beta}$.
There exist isogenous elliptic curves ${\EC}/\Q$ and ${\EC}'/\Q$ with ${\EC}(\Q)[2] \simeq {\EC}'(\Q)[2] \simeq \Z/2$ and conductor $2^m p$ of the form
\begin{align*}
{\EC} \colon y^2 & =x^3 + ax^2 + bx\\
{\EC}' \colon y^2 &= x^3 - 2ax^2 + (a^2-4b)x.
\end{align*}
in the cases described by Table~\ref{table:Ivorra-type-classification}.

\begin{table}[h!]
    \centering
    \begin{tabular}{| c || c | c | c | c | c | c | c | c |}
         \hline
         \textrm{Label/Type} & $\beta$ & $a$ & $b$  & $\omega(a^2-4b)$ & $\omega(b)$ & bound on $k$ & possible $m$\\
         \hline   
         \hline            
         $\mathbf{I}$ & $p-2^k$ & $\pm \alpha$ & $\ -2^{k-2}$ & $\omega(p)=1$ & $\leq1$ &$2 \leq k \leq 5$\tablefootnote{More precisely, for Type \textbf{I} we have $k \in \{2,4,5\}$ for $+\alpha$ and $k \in \{2,3\}$ for $-\alpha$.} & 2,3,4,5\\
         \hline
         $\mathbf{II}$ & $p+2^k$ & $\pm \alpha$ & $\ 2^{k-2}$ & $\omega(p)=1$ & $ \leq 1$&$k \in \{3,5\}$\tablefootnote{More precisely, for Type \textbf{II}, $-\alpha$ is only permitted if $k=3$.} & 3,5\\
         \hline
         $\mathbf{III}$ & $p-2^k$ & $-\alpha$ & $\ -2^{k-2}$ & $\omega(p)=1$ & $ \leq 1$&$4 \leq k \leq f(p)$ & 4\\
         \hline
         $\mathbf{IV}$ & $p+2^k$ & $-\alpha$ & $\ 2^{k-2}$ & $\omega(p)=1$ & $\leq 1$&$4 \leq k \leq f(p)$ & 4\\
         \hline
         $\mathbf{V}$ & $2^k-p$ & $-\alpha$ & $\ 2^{k-2}$ & $\omega(-p)=1$ & $\leq  1$&$4 \leq k \leq f(p)$ & 4\\
         \hline
         $\mathbf{VI}$ & $p-1$ & $\pm 2\alpha$ & $\ -1$ & $\omega(4p)=2$ & $0$&  & 5\\
         \hline
         $\mathbf{VII}$ & $p-1$ & $\pm 2\alpha$ & $\ p$ & $\omega(-4)=1$ & $1$& & 6\\
         \hline
         $\mathbf{VIII}$ & $p-2^k$ & $\pm 2\alpha$ & $\ -2^k$ & $\omega(4p)=2$ & $1$&$2 \leq k \leq f(p)$ & 6\\
         \hline
         $\mathbf{IX}$ & $p+2^k$ & $\pm 2\alpha$ & $\ 2^k$ & $\omega(4p)=2$ & $1$&$2 \leq k \leq f(p)$ & 6\\ 
         \hline
         $\mathbf{X}$ & $2^k-p$ & $\pm 2\alpha$ & $\ 2^k$ & $\omega(-4p)=2$ & $1$&$2 \leq k \leq f(p)$ & 6\\
         \hline
         $\mathbf{XI}$ & $2p^k-1$ & $\pm 2\alpha$ & $\ -1$ & $\omega(8p^k)=2$ & $0$&$k \in \{1,2\}$ & 7\\
         \hline
         $\mathbf{XII}$ & $2p^k-1$ & $\pm 2\alpha$ & $\ 2p^k$ & $\omega(-4)=1$ & $2$&$k \in \{1,2\}$ & 7\\
         \hline
         $\mathbf{XIII}$ & $p^k+2$ & $\pm 2\alpha$ & $\ p^k$ & $\omega(8)=1$ & $1$&$1 \leq k \leq 164969$ & 7\\
         \hline
         $\mathbf{XIV}$ & $p^k+2$ & $\pm 2\alpha$ & $\ 2$ & $\omega(4p^k)=2$ & $1$&$1 \leq k \leq 164969$ & 7\\
         \hline
         $\mathbf{XV}$ & $p-2$ & $\pm 2\alpha$ & $\ p$ & $\omega(-8)=1$ & $1$&  & 7\\
         \hline
         $\mathbf{XVI}$ & $p-2$ & $\pm 2\alpha$ & $\ -2$ & $\omega(4p)=2$ & $1$&   & 7\\
         \hline
         $\mathbf{XVII}$ & $(p^k-1)/2$ & $\pm 4\alpha$ & $\ -2$ & $\omega(8p^k)=2$ & $1$&$k \in \{1,2\}$& 8\\
         \hline
         $\mathbf{XVIII}$ & $(p^k-1)/2$ & $\pm 4\alpha$ & $\ 2p^k$ & $\omega(-8)=1$ & $2$&$k \in \{1,2\}$ & 8\\
         \hline
         $\mathbf{XIX}$ & $(p^k+1)/2$ & $\pm 4\alpha$ & $\ 2$ & $\omega(8p^k)=2$ & $1$&$k \in \{1,2\}$ & 8\\
         \hline
         $\mathbf{XX}$ & $(p^k+1)/2$ & $\pm 4\alpha$ & $\ 2p^k$ & $\omega(8)=1$ & $2$& $k \in \{1,2\}$ & 8\\
         \hline
    \end{tabular}\caption{Our classification of most of the elliptic curves described in \cite{ivorra2004courbes}.}
    \label{table:Ivorra-type-classification}
\end{table}

\end{theorem}

\begin{definition}
We refer to any of the curves arising in Theorem~\ref{thm:general-ivorra-type} as \textbf{Ivorra Curves}.
\end{definition}

\begin{remark}
For an explicit correspondence between our classification types and the curves listed in \cite{ivorra2004courbes}, please see Table~\ref{table:dictionary} in Appendix~\ref{appendix-dictionary}.
\end{remark}

\begin{remark}
Sage \cite{sage} code for computing Ivorra curves can be found at \cite{ivorra-sage-code}.
\end{remark}

\section{Bounding ranks of Ivorra curves}

\subsection{}
We begin by providing an upper bound for the algebraic ranks of the elliptic curves arising from Theorem~\ref{thm:general-ivorra-type}.
First, recall that for $2$-isogenous elliptic curves
\begin{align*}
{\EC} \colon y^2 & =x^3 + ax^2 + bx\\
{\EC}' \colon y^2 &= x^3 - 2ax^2 + (a^2-4b)x,
\end{align*}
we have $r=\mathrm{rank}_\Z {\EC}(\Q) = \mathrm{rank}_\Z {\EC}'(\Q)$.
The general method of $2$-descent provides the following naive upper bound on $r$.

\begin{proposition}\label{prop:naive-rank-bound-aguirre}
Keep the notation introduced above.
Then
\begin{equation}\label{eq:naive-rank-bound}
r \leq \omega(a^2-4b) + \omega(b)-1.
\end{equation}
More generally, let $\EC/\Q$ be any elliptic curve with a non-trivial point of $2$-torsion and let $a$ (resp. $m$) be number of primes of additive (resp. multi-
plicative) bad reduction of $\EC/\Q$. Then:
\begin{equation}\label{eq:naive-rank-bound-additive-mult}
r \leq m + 2a - 1.
\end{equation}
\end{proposition}

\begin{proof}
This is \cite[Proposition~1.1]{aguirre2008elliptic}.
\end{proof}

It follows immediately that we have $r \leq 1$ for the elliptic curves of Types \textbf{I} through \textbf{VII} and also \textbf{XI}, \textbf{XIII}, and \textbf{XV}.
For the rest of the families, we get $r \leq 2$, but by actually performing the $2$-descent, we can \textit{often} improve this bound to $r \leq 1$.
While the proof of the following theorem is a bit tedious, we have made every effort to streamline the arguments and their presentations.

\begin{theorem}
\label{thm:rank-bound-for-ivorra-curves}
Let ${\EC}$ and ${\EC}'$ be a pair of elliptic curves of type $T$ arising in Theorem~\ref{thm:general-ivorra-type} and let $r=\mathrm{rank}_\Z {\EC}(\Q) = \mathrm{rank}_\Z {\EC}'(\Q)$.
Then $r \leq 1$ in the following cases.
\begin{enumerate}
\item[\textup{(1)}]  $T \in \{ \mathbf{I, II, III, IV, V, VI, VII, XI, XIII, XV}\}$\label{thm:case-easy}
\item[\textup{(2)}] $T =\mathbf{VIII}$ and $k=2$\label{thm:VIIIk2}
\item[\textup{(3)}]  $T =\mathbf{X}$
\label{thm:case-X} 
\item[\textup{(4)}] $T=\mathbf{XII}$ with $k=2$ and $-64$ is not a fourth power mod $p$
\item[\textup{(5)}] $T=\mathbf{XVI}$\label{thm:case-XVI}
\item[\textup{(6)}] $T=\mathbf{XVII}$ with $k=1$ and $p \equiv 3 \pmod 8$\label{thm:case-XVII} 
\item[\textup{(7)}] $T=\mathbf{XVIII}$ with $k=1$
\label{thm:case-XVIII} 
\item[\textup{(8)}] $T=\mathbf{XIX}$ and $p \not\equiv 1 \pmod 8$.
\label{thm:case-XIX} 
\end{enumerate}
In fact, in case (6), the rank is always 0.
\end{theorem}

\begin{proof}
Here, we give the proof when $T=\mathbf{X}$ to illustrate the technique.
We refer the reader to Appendix~\ref{appendix} for the full proof.

Consider the curves
\begin{align*}
{\EC} \colon & y^2 =x^3 + ax^2 + bx \\
{\EC}' \colon & y^2 = x^3 - 2ax^2 + (a^2-4b)x
\end{align*}
of type $T=\mathbf{X}$ for some odd prime $p$.
Thus there is some prime $p$ and some integer $2 \leq k \leq f(p)$ such that $2^k - p$ is a perfect square, with 
\begin{align*}
a&=\pm 2 \sqrt{2^k - p},\\ 
b&=2^k, \quad \text{and}\\
a^2-4b&=-4p.
\end{align*}
Let us first observe that if $k=2$, then for $2^k-p$ to be a perfect square we must have $p \equiv 3 \pmod 8$, and if $k \geq 3$ then we must have $p \equiv 7 \pmod 8$.

Both ${\EC}(\Q)[2]$ and ${\EC}'(\Q)[2]$ consist of $K=\{ \mathcal{O}, (0,0) \}$, and there is a $2$-isogeny $\phi \colon {\EC} \to {\EC}'$ with kernel $K$.
Let $\phi'$ denote the dual isogeny.
Since ${\EC}$ and ${\EC}'$ are $2$-isogenous, we have 
\[
\mathrm{rank}_\Z {\EC}(\Q)=\mathrm{rank}_\Z {\EC}'(\Q).
\]

To bound these ranks, we use $2$-descent.
Each curve has conductor $2^6 p$.
Let $\Sigma =\{\pm 1, \pm 2, \pm p, \pm 2p\}$.
To each $d \in \Sigma$ we have the associated homogeneous spaces
\begin{align*}
C_d \colon & dw^2 = d^2-2 a d z^2 + (a^2-4b) z^4 \\
C_d' \colon  & dw^2 = d^2+4 a d z^2 -16b z^4.
\end{align*}
For $\phi$ (and analogously for $\phi'$), we obtain Selmer groups
\[
\Sel^{(\phi)}({\EC}/\Q) = \{ d\in \Sigma \ | \ C_d(\Q_v) \neq \emptyset \text{ for every place } v \in \Sigma \}.
\]
In particular, $\Sel^{(\phi)}({\EC}/\Q)$ is a $2$-group.
There are injections 
\begin{equation*}
\label{eq:selmer-injection-X}
{\EC}'(\Q) /  \phi({\EC}(\Q)) \xhookrightarrow{\delta} \Sel^{(\phi)}({\EC}/\Q) \quad \text{and} \quad  {\EC}(\Q) /  \phi'({\EC}'(\Q)) \xhookrightarrow{\delta'} \Sel^{(\phi')}({\EC}'/\Q)
\end{equation*}
with the explicit values
\[
\delta(\mathcal{O}) = 1 \quad \text{and} \quad
\delta((0,0)) = \text{the square-free part of } a^2-4b = -p,
\]
and
\[
\delta(\mathcal{O}) = 1 \quad \text{and} \quad
\delta((0,0)) = \text{the square-free part of}\ 2^k,
\]
We record this information in Table~\ref{table: 2-descent-X}, along with everything else we deduce in the rest of the descent argument.

By \cite[Equation (5)]{aguirre2008elliptic}, we have
\begin{equation}\label{eq:selmer-rank-bound-X}
r \leq \dim_{\mathbb{F}_2} \Sel^{(\phi)}({\EC}/\Q) + \dim_{\mathbb{F}_2} \Sel^{(\phi')}({\EC}'/\Q) - 2.
\end{equation}

We now compute these Selmer groups by studying the homogeneous spaces defined above.

\begin{center}
\begin{table}[H]
\caption{This table summarizes the findings from performing $2$-descent on elliptic curves from family $\mathbf{X}$, depending on the parity of $k$.
In each case, the first row gives information about $\Sel^{(\phi)}({\EC}/\Q)$, and the second row gives information about $\Sel^{(\phi')}({\EC'}/\Q)$.
A green cell indicates a global point in the corresponding Selmer group coming from $2$-torsion, while a red cell indicates that the corresponding homogeneous space has no local solutions over the indicated field, or that we may use the group structure of the Selmer group to deduce that it does not contain this element.
Finally, a blue cell indicates that we do not need to analyze this cell in order to obtain our bound on the Mordell-Weil rank.}
\label{table: 2-descent-X}
\begin{tabular}{|c|c|c|c|c|c|c|c|c|}
\hline
T                               & \textbf{$1$}                                               & \textbf{$-1$}                                       & \textbf{$2$}                                        & \textbf{$-2$}                                       & \textbf{$p$}                                       & \textbf{$-p$}                                      & \textbf{$2p$}                                       & \textbf{$-2p$}                            \\ \hline \hline
                       
\multirow{2}{*}{\textbf{X}, $k$ even}                         & \cellcolor{green!25}$\delta(\mathcal{O})$   &          \cellcolor{blue!25}           &               \cellcolor{red!25}$\Q_2$  &\cellcolor{red!25}$\Q_2$     &         \cellcolor{blue!25}           & \cellcolor{green!25}$\delta(0,0)$ &                              \cellcolor{red!25}grp          &   \cellcolor{red!25}grp            \\ \cline{2-9} 
                                          & \cellcolor{green!25}$\delta(\mathcal{O})$ & \cellcolor{red!25}$\R$             &  \cellcolor{blue!25}
                                          & 
                                          \cellcolor{red!25}$\R$             & \cellcolor{red!25}$\Q_p$          & \cellcolor{red!25}$\Q_p$          & \cellcolor{red!25}$\Q_p$           & \cellcolor{red!25}$\Q_p$ \\ \hline \hline

\multirow{2}{*}{\textbf{X}, $k$ odd}                         & \cellcolor{green!25}$\delta(\mathcal{O})$ &       \cellcolor{blue!25}              &    \cellcolor{red!25}$\Q_2$        &  \cellcolor{red!25}$\Q_2$        &          \cellcolor{blue!25}          & \cellcolor{green!25}$\delta(0,0)$ & \cellcolor{red!25}grp  &   \cellcolor{red!25}grp            \\ \cline{2-9} 
                                          & \cellcolor{green!25}$\delta(\mathcal{O})$ & \cellcolor{red!25}$\R$             & \cellcolor{green!25}$\delta(0,0)$ & \cellcolor{red!25}$\R$             & \cellcolor{red!25}$\Q_p$          & \cellcolor{red!25}$\Q_p$          & \cellcolor{red!25}$\Q_p$           & \cellcolor{red!25}$\Q_p$ \\ \hline

\end{tabular}
\end{table}
\end{center}

\noindent\textbf{(i)} \fbox{$d=-1 :$} Consider the homogeneous space
\begin{align*}
C'_{-1} \colon & -w^2 = 1- 4 a z^2 -16b z^4.
\end{align*}
The left-hand side is always non-positive, while the right-hand side certainly takes a positive value when $z=0$.
Viewing the right-hand side as a quadratic in $z^2$, the discriminant is 
\[
16a^2 - 4 \cdot 16 b = 16(a^2-4b)=16(-4p) <0,
\]
which shows that $C'_{-1}(\R)=\emptyset$, and so $-1 \not\in \Sel^{(\phi')}({\EC}'/\Q)$.
This is recorded in the table by filling the corresponding cell red and labeling it with $\R$.

\smallskip

\noindent \textbf{(ii)} \fbox{$d=\pm 2$:}
For curves of type \textbf{X}, the homogeneous spaces corresponding to $d=\pm 2$ specialize to
\[
C_{\pm2} \colon  \pm2w^2 = 4 \mp 8\alpha z^2 -4p z^4.
\]
(Note that there is some ambiguity on the sign of $8\alpha z$, since there is a choice of sign for both $d$ and $a$, but the argument which follows is insensitive to this sign.)

If there is a solution with $w,z \in \Q_2$, then writing $\ord_2(z)=j$ we have
\[
\ord_2(\mathrm{RHS}) \geq \min\{2, 3+2j, 2+4j\}
\]
with equality except, perhaps, when two of the values are equal.
Thus if $j<0$ or $j \geq 1$, then $\ord_2(\mathrm{RHS})$ is even, which is impossible.
If $j=0$ then we can have $\ord_2(\mathrm{RHS})=3$, so we conclude that $\ord_2(w)=1$ and $j=0$.

Write $w=2W$ with $W \in \Z_2^\times$.
Substituting and simplifying yields
\begin{align*}
\pm 2W^2 & =1 \mp 2 \alpha z^2 -pz^4, \quad \text{or} \\
\pm 2W^2 \pm 2 \alpha z^2 & =1  -pz^4, \quad \textrm{or} \\
\pm 2(W^2 \pm \alpha z^2) & = 1 - pz^4.
\end{align*}
Now squaring both sides we observe that
\[
4(W^4 \pm 2\alpha z^2 + \alpha^2 z^2) = 1 -2p z^4 + p^2 z^8
\]
Since $1$ is the only odd square mod 8, and using the fact that $\alpha^2= 2^k-p$ with $k>0$, we have
\[
4(1 - p) \equiv 1 - 2p +1  \equiv 2(1-p)\pmod 8.
\]
But since $p \equiv 3$ or $7 \pmod 8$, this yields a contradiction, as the left-hand side is zero while the right-hand side is nonzero.
We have thus shown that $C_{\pm 2}(\Q_2)=\emptyset$, so $\pm 2 \not\in \Sel^{(\phi)}({\EC}/\Q)$.

Now let us consider the homogeneous space
\[
C'_{-2} \colon  -2w^2 = 4 - 8 a z^2 -16b z^4.
\]
The left-hand side is always non-positive, while the right-hand side certainly takes positive values.
Viewed as a quadratic in $z^2$, the discriminant on the right-hand side is $64(a^2-4b)=16(-4p)<0$, which shows that $C'_{-2}(\R)=\emptyset$, so $- 2 \not\in \Sel^{(\phi')}(\EC'/\Q)$.

\smallskip

\noindent \textbf{(iii)} \fbox{$d=\pm p$:}
Consider the homogeneous spaces
\[
C'_{\pm p} \colon \pm p w^2 = p^2 \pm 4p a z^2 -16b z^4.
\]
We have $p \nmid ab$.
Suppose there is a solution $(w,z) \in C'_{\pm p}(\Q_p)$.
Then $\ord_p(pw^2)$ is odd and 
\[
\ord_p(\mathrm{RHS}) \geq \min \{ 2, 1+2j, 4j \}
\]
with $j = \ord_p(z)$.
Since $4j$ is even, we must have $w, z \in \Z_p$.
Reducing mod $p$ shows that $z \in p\Zp$, and then reducing mod $p^2$ shows that $w \in p\Zp$, but then this implies $p^2 \equiv 0 \pmod p^3$, a contradiction.

If we consider the form of the homogeneous spaces $C'_{\pm 2p}$, we see that the exact same argument applies again, and we deduce that $\pm p, 2p \not\in \Sel^{(\phi')}(\EC'/\Q)$.

\smallskip

\noindent \textbf{(iv)} \fbox{Group structure obstructions:} Recall that our Selmer groups have been identified with a subgroup of $\Q^\times / (\Q^\times)^2$.
It is now possible to use this group structure to eliminate a few more cases, as indicated in the table.
For instance, when $k$ is even, we have $-p \in \Sel^{(\phi)}(\EC/\Q)$ and $\pm 2 \notin \Sel^{(\phi)}(\EC/\Q)$.
If $2p \in \Sel^{(\phi)}(\EC/\Q)$, then also $(2p)(-p)=-2 \in \Sel^{(\phi)}(\EC/\Q)$, so we conclude $2p \not\in \Sel^{(\phi)}(\EC/\Q)$, and similarly for $-2p$.

There remain some undetermined cells, which we color in blue, but nevertheless, we have computed enough to show that, regardless of the parity of $k$, we have 
\[
\dim_{\mathbb{F}_2} \Sel^{(\phi)}(\EC/\Q) + \dim_{\mathbb{F}_2} \Sel^{(\phi')}(\EC'/\Q) \leq  3.
\]
So ~\eqref{eq:selmer-rank-bound-X} implies that
\[
\mathrm{rank}_\Z \EC(\Q), \mathrm{rank}_\Z \EC'(\Q) \leq 1
\]
as desired.
\end{proof}

\begin{remark}\label{rmk:more-precise}
It is often possible to give a more precise result.
For instance, the curves of conductor $8p$ and type \textbf{I} have rank $0$ when $p \equiv 9 \pmod{16}$, because one can check that $C_{\pm 2}(\Q_2)=\emptyset$.

Computational evidence suggests that for $T=\mathbf{VIII}$ and $k=2$, the curves have rank 0  (resp. 1) when $+\alpha$ (resp. $-\alpha$) is used.
A similar phenomenon is witnessed in families \textbf{XVI} and \textbf{XIX}.
\end{remark}

\begin{remark}\label{rmk:computational-conjectures}
It is not entirely clear whether any of the families listed in Ivorra's theorem are infinite! For instance, the infinitude of some of the families would follow from special cases of the Bunyakovsky conjecture \cite{Bunyakovsky}, which is currently open.
\end{remark}

\begin{corollary}\label{cor:watkins-holds-for-ivorra-type}
Let $\EC$ be an elliptic curve as in the statement of Theorem~\ref{thm:rank-bound-for-ivorra-curves}.
Assuming either the Birch and Swinnerton-Dyer Conjecture \textit{or} that $\Sha(\EC/\Q)$ is finite, Watkins' Conjecture holds for $\EC$.
\end{corollary}

\begin{proof}
If $\mathrm{rank}_\Z \EC(\Q)=0$ the result is trivial, otherwise we have $\mathrm{rank}_\Z \EC(\Q)=1$.
Assuming either BSD or finiteness of $\Sha(\EC/\Q)$, the analytic rank of $\EC/\Q$ is also $1$.
By the contrapositive of \cite[Theorem~1.1]{CE09} 
the modular degree of $\EC$ is divisible by $2$.
Hence Watkins' Conjecture is satisfied.
\end{proof}

\subsection{Observations}
In this section, we provide examples of several families of Ivorra curves where the rank can indeed be 2.
In other words, we show that the hypotheses in Theorem~\ref{thm:rank-bound-for-ivorra-curves} are necessary.
We make remarks about the families that appear to be finite, and we can verify their low ranks directly using Magma~\cite{Magma}.

\begin{enumerate}
\item[(a)] 
\textbf{XII, $k=2$:} Using Magma to check primes up to $10^9$, the only primes falling into this category were $p=5, 29, 5741,$ and $33461$.
For all four of these primes, we get two pairs of elliptic curves (\cite[Th\'{e}or\`{e}me 7 (1)]{ivorra2004courbes}), and the first pair has rank 1 while the second pair has rank 0.

\item[(b)]  \textbf{XVII, $k=2$:} Searching up to $10^9$ produces only four primes falling into this family: $3, 17, 577$, and $665857$.
The first two yield curves of ranks 0 and 1, respectively.
The third exhibits a rank 2 counterexample which shows the necessity of working with $k=1$ in this family.
The fourth gives curves of rank 0 which are not in Cremona's database.

\item[(c)]  \textbf{IX, $k=2$:} Searching up to $10^7$, there appears to be exactly one prime in this family, namely $p=5$, giving \href{https://www.lmfdb.org/EllipticCurve/Q/320c1/}{Cremona 320c1} (rank 0) and \href{http://www.lmfdb.org/EllipticCurve/Q/320f1}{Cremona 320f1} (rank 1).

\item[(d)]  \textbf{XVIII, $k=2$:} These primes are exactly the same ones as family \textbf{XII}, which seems to be finite when $k=2$.
The curves obtained in this way from \cite[Th\'{e}or\`{e}me~8(1)]{ivorra2004courbes} can all be checked to have rank $\leq 1$.

\item[(e)]  \textbf{XIX, $k = 2$:} Up to $10^9$, there are only four such primes: 7, 41, 239, 9369319.
The curves obtained in this way from \cite[Th\'{e}or\`{e}me~8(2)]{ivorra2004courbes} can all be checked to have rank $\leq 1$.

\item[(f)] \label{XX fam} \textbf{XX, $k = 2$:} This is the same list of primes as XIX, $k=2$.
In particular, it appears to be finite.
The curves obtained in this way from \cite[Th\'{e}or\`{e}me~8(2)]{ivorra2004courbes} can all be checked to have rank $\leq 1$.
\end{enumerate}

\begin{table}[h!]
    \centering
    \begin{tabular}{| c || c | p{3cm} | c | c |}
         \hline
         \textrm{Type} & Conditions on $k$, $p$ & rank & curve with rank 2&  $N_{\EC} = 2^m p$ \\
         \hline   
         \hline
         \textbf{VIII} & $k= 2$; $p\equiv 5\pmod{8}$ & $\textrm{rank}\le 1$ by Th~\ref{thm:rank-bound-for-ivorra-curves} & - &  \\
         \hline
         \textbf{VIII} & $k \ge 3$; $p\equiv 1\pmod{8}$ & rank 2 possible & \href{https://www.lmfdb.org/EllipticCurve/Q/7232/a/1}{7232c} &  $2^6 \cdot 113$ \\
         \hline
         \hline
         \textbf{IX} & $k = 2$; $p\equiv 5\pmod{8}$ & see note (c) &  &  \\
         \hline
         \textbf{IX} & $k\ge 3$; $p\equiv 1\pmod{8}$ & rank 2 possible  & \href{https://www.lmfdb.org/EllipticCurve/Q/16448/b/1}{16448j} & $2^6 \cdot 257$ \\
         \hline
         \hline
         \textbf{XII} & $k=1$; $p\equiv 1\pmod{8}$ & rank 2 possible  & \href{https://www.lmfdb.org/EllipticCurve/Q/5248/a/1}{5248a} & $2^7 \cdot 41$\\
         \hline
         \textbf{XII} & $k=2$; $-64 \equiv x^4 \pmod{p}$ & see note (a) &   & \\
         \hline
         \textbf{XII} & $k=2$; $-64 \not\equiv x^4 \pmod{p}$ &  $\textrm{rank}\le 1$ by Th~\ref{thm:rank-bound-for-ivorra-curves}& - & \\
         \hline
         \hline
         \textbf{XIV} & 
         & rank 2 possible  & \href{https://www.lmfdb.org/EllipticCurve/Q/10112/k/}{10112c} & $2^7 \cdot 79$ \\
         \hline
         \hline
         \textbf{XVII} & $k=1$; $p \equiv 1 \pmod 8$ & rank 2 possible & \href{https://www.lmfdb.org/EllipticCurve/Q/18688/c/}{18688b} & $2^8 \cdot 73$\\
         \hline
         \textbf{XVII} & $k=1$; $p \equiv 3 \pmod 8$ &  $\textrm{rank}=0$ by Th~\ref{thm:rank-bound-for-ivorra-curves} & - & \\
         \hline
         \textbf{XVII} & $k = 2$ & see note (b); \newline rank 2 possible & \href{https://www.lmfdb.org/EllipticCurve/Q/147712/c/1}{147712e} &  $2^8 \cdot 577$\\
         \hline
         \hline
         \textbf{XVIII} & $k=1$ & $\textrm{rank}\le 1$ by Th~\ref{thm:rank-bound-for-ivorra-curves} &  & \\
         \hline
         \textbf{XVIII} & $k = 2$ &  see note (d) & - & \\
         \hline 
         \hline 
         \textbf{XIX} & $k=1$; $p\equiv 1\pmod{8}$ & rank 2 possible & \href{http://www.lmfdb.org/EllipticCurve/Q/24832d1}{24832d} & $2^8 \cdot 97$ \\
         \hline
         \textbf{XIX} & $k=1$; $p\equiv 7\pmod{8}$ & $\textrm{rank}\le 1$ by Th~\ref{thm:rank-bound-for-ivorra-curves} & -  & \\
         \hline
         \textbf{XIX} & $k = 2$; $p \equiv 1 \pmod 8$ & see note (e) & - &  \\
         \hline
         \textbf{XIX} & $k = 2$; $p \not\equiv 1 \pmod 8$ & see note (e); \newline
         $\textrm{rank}\le 1$ by Th~\ref{thm:rank-bound-for-ivorra-curves} & - &  \\
         \hline  
         \hline
         \textbf{XX} & $k=1$; $p\equiv 1\pmod{8}$ & rank 2 possible & \href{https://www.lmfdb.org/EllipticCurve/Q/86272/b/1}{86272a} & $2^8 \cdot 337$\\
         \hline
         \textbf{XX} & $k=1$; $p\equiv 7\pmod{8}$ & rank 2 possible & \href{https://www.lmfdb.org/EllipticCurve/Q/7936/a/1}{7936b} & $2^8 \cdot 31$\\
         \hline
         \textbf{XX} & $k = 2$ & see note (f) & - & \\
         \hline         
         \end{tabular}     
    \label{table: counter examples}
\caption{
This table explores the families appearing in Theorem~\ref{thm:rank-bound-for-ivorra-curves} with extra hypotheses, or families which are omitted entirely.
In particular, it includes explicit examples of Ivorra curves of rank $2$.
Note that the requirement that $\beta$ is a square puts congruence restrictions on $p$ modulo 8.
}
\end{table}

\section{Watkins' Conjecture in Quadratic Twist Families}

Let $\EC$ be an elliptic curve of conductor $N_{\EC} = N = 2^m p$ arising in Theorem~\ref{thm:rank-bound-for-ivorra-curves}.
For any odd prime $q$ and $d= \pm q$,  we let $\EC^{(d)}$ denote the corresponding quadratic twist.
By \cite[p.~675]{delaunay2003computing}), its conductor is given by $N^{(d)} = 2^k pq^2$, where $k\geq m$ with equality when $d \equiv 1 \pmod 4$.

\begin{remark}\label{rmk:twist-power}
In fact, even when $q \equiv 3 \pmod 4$, it is still sometimes the case that $k=m$.
For example, if $\EC$ is the elliptic curve with Cremona label \href{https://www.lmfdb.org/EllipticCurve/Q/5248/a/1}{5248a2}, which has conductor $N = 2^7 \cdot 41$, and if we take $q=7$, then the corresponding quadratic twist $\EC^{(q)}$ is the curve with Cremona label \href{https://www.lmfdb.org/EllipticCurve/Q/257152/bg/1}{257152bg2} and conductor $N^{(d)} = 2^7 \cdot 7^2 \cdot 41$.
So in this case, $k=m=7$.

On the other hand, if $\EC$ is the elliptic curve with Cremona label \href{https://www.lmfdb.org/EllipticCurve/Q/692a2/}{692a1}, which has conductor $N = 2^2 \cdot 173$, and if we take $q=7$, then the corresponding quadratic twist $\EC^{(q)}$ is the curve with Cremona label \href{https://www.lmfdb.org/EllipticCurve/Q/135632b1/}{135632c1} and conductor $N^{(d)} = 2^4 \cdot 7^2 \cdot 173$.
So in this case, $m=2$ while $k=4$.
\end{remark}

Before proceeding, we also note that the non-triviality of $\EC(\Q)[2]$ has a useful consequence.
Denote by $a_q(\EC) = q+1 -\#\tilde{\EC}(\mathbb{F}_q)$.

\begin{Lemma}
\label{even aq}
Let $\EC$ be an Ivorra curve of type $\mathbf{I}$ through $\mathbf{XX}$ and $q$ be a prime of good reduction.  Then $a_q(\EC)$ is even.
\end{Lemma}

\begin{proof}
This follows from the fact that the $|\EC(\Q)[2]|=2$ for every Ivorra curve of type \textbf{I} through \textbf{XX}, and that for every prime $q$ of good reduction there is an injection
\[
\EC(\Q)_\mathrm{tors} \hookrightarrow \widetilde{\EC}(\mathbb{F}_q).
\]
\end{proof}

\subsection{Modular form preliminaries}\label{sec:modular-preliminaries}

Let $f \in S_2(\Gamma_0(N))$ be a weight 2 cuspidal holomorphic modular form.
Let $\mathfrak{h}$ denote the upper half-plane in $\mathbb{C}$.

\begin{definition} The \textbf{Petersson norm} of $f \in S_2(\Gamma_0(N))$ is given by
\[
\Vert f\Vert_{N} = \left( \int_{\Gamma_0(N)\backslash \mathfrak{h}} |f(z)|^2 dx \wedge dy \right)^{1/2}, \quad z=x+iy \ \text{and} \ y>0.
\]
\end{definition}

We are interested in those $f \in S_2(\Gamma_0(N))$ which are associated to elliptic curves.
In particular, let $\EC/\Q$ be an Ivorra curve and $\omega_\EC$ its N\'{e}ron differential.
Recall from \eqref{eq:modular-param} that we have a paramaterization
\[
X_0(N) \xrightarrow{\phi_{\EC}} \EC.
\]
Assume for the moment that $\EC$ is an \textit{optimal} elliptic curve, in the sense that $\mathrm{deg}~\phi_\EC$ (equivalently $m_\EC$) is minimal in its isogeny class.
Then $\phi_\EC^\ast\omega_\EC$ is a regular differential on $X_0(N)$, and 
\begin{equation}\label{eq:manin-relation}
\phi_\EC^\ast \omega_\EC = 2 \pi i c f_\EC(z)dz,
\end{equation}
where $c$ is a unique integer up to sign; see \cite[Proposition~2]{edixhoven-manin}.

\begin{definition} The \textbf{manin constant} of $\EC / \Q$ is given by $c_\EC=|c|$, where $c$ is the integer  in \eqref{eq:manin-relation}.
\end{definition}

Now suppose that $\EC'$ is another elliptic curve and $E \xrightarrow{\psi} \EC'$ is an isogeny.
We obtain a modular parameterization $$X_0(N) \xrightarrow{(\psi \circ \phi)} \EC',$$ and we similarly have 
\[
(\psi \circ \phi_\EC)^\ast \omega_\EC' = 2 \pi i c' f_\EC(z)dz.
\]
In this case, we define the Manin constant of $\EC'$ to be $c_{\EC'}=|c'|$.
We have 
\[
c' = \delta c
\]
where $\delta$ is an integer which divides $\mathrm{deg}~\psi$.
In particular, given an isogeny class with two $2$-isogenous curves $(\EC,\EC')$ with $\EC$ optimal, it is expected that $c_\EC=1$ and that $c_{\EC'} \in \{1,2\}$.
Numerical computations give plenty of examples to show that both possibilities for $c_{\EC'}$ occur.

\begin{remark}\label{rmk:manin-constant-1}
It is conjectured that $c_\EC=1$ for optimal curves or
whenever the rank of $\EC$ is positive.
For a nice summary of what is currently known in this direction, see \cite{ARS-manin}.

\end{remark}

\begin{example}\label{example:big-manin}
Consider the isogeny class \href{https://www.lmfdb.org/EllipticCurve/Q/116c2/}{116c} containing the two curves 
\begin{align*}
\EC &\colon y^2 = x^3 - 10x^2+29x, \\ 
\EC' &\colon y^2 = x^3 + 5x^2 -x.
\end{align*}
These curves belong to family \textbf{I} with $p=29$; more precisely, they arise as curves A2 and A1, respectively, in \cite[Th\'{e}or\`{e}me~2]{ivorra2004courbes}.
Using LMFDB, one verifies that these curves have Mordell--Weil rank 0 and
\begin{align*}
m_\EC = 3 \cdot 5, &\quad c_\EC = 1 \\
m_{\EC'}=2 \cdot 3 \cdot 5, &\quad   c_{\EC'}=2.
\end{align*}
We note that Watkins' conjecture holds (vacuously) for each of these curves.
\end{example}

\begin{example}\label{example:big-manin 2}
Consider the isogeny class \href{https://www.lmfdb.org/EllipticCurve/Q/328/a/}{328a} containing the two curves 
\begin{align*}
\EC &\colon y^2 = x^3 - 3x^2-8x, \\ 
\EC' &\colon y^2 = x^3 + 6x^2 +41x.
\end{align*}
These curves belong to family \textbf{I} with $p=41$; more precisely, they arise as curves B1 and B2, respectively, in \cite[Th\'{e}or\`{e}me~3]{ivorra2004courbes}.
Using LMFDB, one verifies that these curves have Mordell--Weil rank 1 and
\begin{align*}
m_\EC = 2^3, &\quad c_\EC = 1,\\
m_{\EC'}=2^5, &\quad c_{\EC'}=1.
\end{align*}
Note that Watkins' conjecture holds for each of these curves, this time for less trivial reasons.
\end{example}

The constants $m_\EC, c_\EC$, and $\Vert f \Vert_{N}^2$ are all closely related, as we will see in the next section.

\subsection{Quadratic twists of Ivorra curves}
\label{sec:petterson-norm}

Let $\EC: y^2 = f(x)$ be an elliptic curve, then its quadratic twist by an integer $d$ is the curve $\EC^{(d)}: dy^2 = f(x)$.
Given $\EC/\Q$,  Goldfeld's Conjecture predicts that $50\%$ of its quadratic twists have (analytic) rank 0 and $50\%$ of its quadratic twists have rank 1.
The remaining $0\%$ (but still infinitely many) of its quadratic twists have rank $\ge 2$.

\begin{proposition}
Assume that Goldfeld's Conjecture is true.
Then Watkins' Conjecture is true for 100\% of the quadratic twists $\EC^{(d)}$.
\end{proposition}

\begin{proof}
Assuming Goldfeld's Conjecture, 100\% of the twists $\EC^{( d)}$ have analytic rank $\le 1$.
When the rank is zero, the result is trivial, and when the rank is 1, Watkins' Conjecture holds by \cite[Theorem~1.1]{CE09}.
\end{proof}
\noindent We spend the rest of this section giving \textit{unconditional} results in this direction.

Let $f_{\EC}$ be the modular form associated with the elliptic curve $\EC$ of conductor $N$.
As in the previous section, we write $\Vert f_\EC \Vert_{N}$ for the Petersson norm of this elliptic curve, we write $m_{\EC}$ for its modular degree, and we write $c_{\EC}$ for its Manin constant.
Let $q$ be an odd prime and $d= \pm q$.
We write $m_{\EC^{(d)}}, c_{\EC^{(d)}}$, and $\Vert f_{\EC^{(d)}} \Vert_{N^{(d)}}$ for the constants associated to the quadratic twist ${\EC^{(d)}}$.

Using \cite[Theorem~1]{delaunay2003computing}, we see that
\[
\frac{\Vert f_{\EC^{(d)}}\Vert_{N_{\EC^{(d)}}}^2}{\Vert f_{\EC}\Vert_{N}^2} = 2^{k-m}\left(\frac{(q-1)(q+1 - a_q(\EC))(q+1 + a_q(\EC))}{q}\right).
\]
Now, using \cite[(1)]{caro2022watkins} and the fact that $\abs{\Delta_{\EC^{(d)}}} = q^6 \abs{\Delta_{\EC}}$ we deduce that
\begin{align*}
2^{k-m}\left(\frac{(q-1)(q+1 - a_q(\EC))(q+1 + a_q(\EC))}{q}\right) &=\frac{\Vert f_{\EC^{(d)}}\Vert_{N_{\EC^{(d)}}}^2}{\Vert f_{\EC}\Vert_{N}^2}\\
& = \frac{m_{\EC^{(d)}}}{m_{\EC}} \times \frac{c^2_{\EC}}{c^2_{\EC^{(d)}}} \times \abs{\frac{\Delta_{\EC}}{\Delta_{\EC^{(d)}}}}^{1/6} \\
& = \frac{m_{\EC^{(d)}}}{m_{\EC}} \times \frac{c^2_{\EC}}{c^2_{\EC^{(d)}}} \times \frac{1}{q} 
\end{align*}
In particular, we have the relation
\begin{equation}\label{eq:constants-relation}
2^{k-m}{(q-1)(q+1 - a_q(\EC))(q+1 + a_q(\EC))} = \frac{m_{\EC^{(d)}}}{m_{\EC}} \times \frac{c^2_{\EC}}{c^2_{\EC^{(d)}}}.
\end{equation}

The ranks of the quadratic twists $\EC^{(d)}$ are bounded .

\begin{Lemma}\label{lem:rank-4-bound}
Let $\EC$ be any Ivorra curve.
Then for any prime $q$, the quadratic twists $\EC^{(q)}$ and $\EC^{(-q)}$ have rank at most 4.
\end{Lemma}

\begin{proof}
Since $\EC \colon y^2 = f(x)$ is an Ivorra curve, it has a non-trivial $2$-torsion point $P=(X,0)$ with $X \in \Q$.
The quadratic twist is given by $\EC^{(d)} \colon dy^2 = f(x)$, so we also have $P \in \EC^{(d)}(\Q)[2] \neq 0$.
Since $\EC^{(d)}$ has additive reduction at $2$ and $q$, multiplicative reduction at $p$, and good reduction everywhere else, the result now follows from the bound \eqref{eq:naive-rank-bound-additive-mult}.
\end{proof}

In fact, we can do better when $T \in \{ \mathbf{I, II, III, IV, V, VI, VII, XI, XIII, XV}\}$.

\begin{Lemma}\label{lem:trick-rank-leq-3}
Let $\EC$ be an elliptic curve of type $T \in \{ \mathbf{I, II, III, IV, V, VI, VII, XI, XIII, XV}\}$.
Then $\operatorname{rank}_{\Z}(\EC^{(d)}) \leq 3$ where $d\in\{q,-q\}$.
\end{Lemma}

\begin{proof}
Consider the elliptic curves to be written in the form
\begin{align*}
\EC: y^2 &= x^3 + ax^2 + bx\\
\EC': y^2 &= x^3 - 2ax^2 + (a^2 -4b)x.
\end{align*}
By \cite[Proposition~1.1]{aguirre2008elliptic}, and the fact that isogenous elliptic curves have the same Mordell--Weil rank, we know that
\[
\operatorname{rank}_{\Z}(\EC') = \operatorname{rank}_{\Z}(\EC) \leq \omega(a^2 -4b) + \omega(b) -1.
\]
The quadratic twist of $\EC$ and $\EC'$ by $q$ is given by,
\begin{align*}
\EC^{(d)}: y^2 &= x^3 + aq x^2 + bq^2 x\\
\EC'^{(d)}: y^2 &= x^3 - 2aq x^2 + (a^2 - 4b)q^2 x.
\end{align*}
Once again using \cite[Proposition~1.1]{aguirre2008elliptic} we obtain
\begin{align*}
\operatorname{rank}_{\Z}(\EC^{(d)}) &\leq \omega((a^2-4b)q^2) + \omega(bq^2) - 1\\
& = \omega(a^2 - 4b) + 1 + \omega(b) + 1 -1\\
& = \omega(a^2 - 4b)  + \omega(b) + 1.
\end{align*}
One may argue similarly for $\operatorname{rank}_{\Z}(\EC'^{(d)})$.
The result now follows by direct check (see Table~\ref{table:Ivorra-type-classification}).
\end{proof}

We now leverage the relationship given by equation~\eqref{eq:constants-relation} to prove Watkins' Conjecture for some quadratic twists of Ivorra curves.
We begin with the families of curves covered by Lemma~\ref{lem:trick-rank-leq-3}.

\begin{Lemma}
\label{Lemma 4.3}\label{lem:watkins-trick-twists}
Let $\EC$ be an elliptic curve of type $T \in \{ \mathbf{I, II, III, IV, V, VI, VII, XI, XIII, XV}\}$ 
and further suppose that  $c_{\EC}$ is equal to 1.
Let $q\ge 5$ be a prime of good reduction for $\EC$.
Then Watkin's Conjecture holds for ${\EC^{(d)}}$.
\end{Lemma}

\begin{proof}
Since $a_q(\EC) = 2\theta$ by Lemma~\ref{even aq}, we know that
\[
2^{k-m}{(q-1)(q+1 - 2\theta)(q+1 + 2\theta)} m_\EC = {m_{\EC^{(d)}}} \times \frac{1}{c^2_{\EC^{(d)}}}.
\]
The LHS of the above equation is divisible by $2^3$.
Thus, Watkin's Conjecture holds for $\EC^{(d)}$, since by Lemma~\ref{lem:trick-rank-leq-3} the rank of $\EC^{(d)}$ is at most 3.
\end{proof}

\begin{Lemma}
\label{Lemma 4.4}
Let $\EC$ be an elliptic curve of type $T \in \{ \mathbf{I, II, III, IV, V, VI, VII, XI, XIII, XV}\}$ such that $c_{\EC}>1$.
Then Watkin's Conjecture holds for ${\EC^{(d)}}$ in the following cases
\begin{enumerate}
    \item[\textup{(}i\textup{)}] $q\equiv 1\pmod{8}$ is a prime of good reduction of $\EC$.
    \item[\textup{(}ii\textup{)}] $q\equiv 1\pmod{4}$ is a prime of good reduction of $\EC$, the Mordell--Weil rank of $\EC$ is \textit{exactly 1}, and Watkin's Conjecture holds for $\EC$.
    \item[\textup{(}iii\textup{)}] $q\equiv 3\pmod{4}$ is a prime of good supersingular reduction of $\EC$.
\end{enumerate}
\end{Lemma}

\begin{proof}
Recall from Theorem~\ref{thm:rank-bound-for-ivorra-curves} that the Mordell--Weil rank is $\le 1$ in this case.
Since $a_q(\EC) = 2\theta$, we know that
\[
(q-1)(q+1 - 2\theta)(q+1 + 2\theta) m_\EC = {m_{\EC^{(d)}}} \times \frac{4}{c^2_{\EC^{(d)}}}.
\]
On the other hand, when $a_q(\EC) = 0$, we know that
\[
(q-1)(q+1)(q+1) m_\EC = {m_{\EC^{(d)}}} \times \frac{4}{c^2_{\EC^{(d)}}}.
\]
In each of the cases, the LHS is divisible by $2^3$.
Thus, Watkin's Conjecture holds for $\EC^{(q)}$.
\end{proof}

We may use the same technique to deduce Watkins' Conjecture for twists of other Ivorra curves (such as those appearing in Theorem~\ref{thm:rank-bound-for-ivorra-curves}) in many cases.

\begin{Lemma}
\label{Lemma 4.5}
Let $\EC$ be an elliptic curve of conductor $2^m p$ with a non-trivial 2-torsion.
Further suppose that  $c_{\EC}$ is equal to 1.
Then Watkin's Conjecture holds for ${\EC^{(d)}}$ under either of the following conditions :
\begin{enumerate}
    \item[\textup{(}i\textup{)}] $q\ge 5$ is a prime of good supersingular reduction of $\EC$.
    \item[\textup{(}ii\textup{)}] $q\equiv 1\pmod{4}$ is a prime of good ordinary reduction of $\EC$.
\end{enumerate}
\end{Lemma}

\begin{proof}
When $a_q(\EC) = 0$, we know that
\[
2^{k-m}{(q-1)(q+1)(q+1)} m_\EC = {m_{\EC^{(d)}}} \times \frac{1}{c^2_{\EC^{(d)}}}.
\]
If $q\equiv 1\pmod{4}$ then the LHS is divisible by $2^4$ whereas, if $q\equiv 3\pmod{4}$ then the LHS is divisible by $2^5$.
In either case, Watkin's Conjecture holds for $\EC^{(d)}$ by Lemma~\ref{lem:rank-4-bound}.

On the other hand, when $a_q(\EC) = 2\theta$ (but not 0), we know that
\[
{2^{k-m}(q-1)(q+1 - 2\theta)(q+1 + 2\theta)} m_\EC = {m_{\EC^{(d)}}} \times \frac{1}{c^2_{\EC^{(d)}}}.
\]
The condition $q\equiv 1\pmod{4}$ forces the LHS to be divisible by $2^4$, and the result once again follows from Lemma~\ref{lem:rank-4-bound}.
\end{proof}

Putting the lemmas from this section together, we have proven the following theorem.

\begin{theorem}
\label{watkin's for twist using petterson norm}
Let $\EC/\Q$ be an Ivorra curve and $q\geq 5$ be a prime of good reduction of $\EC/\Q$.
Then Watkins' conjecture holds for the quadratic twists $\EC^{(\pm q)}/\Q$ in the following cases.
\begin{enumerate}
\item $\EC$ is of type $T \in \{ \mathbf{I, II, III, IV, V, VI, VII, XI, XIII, XV}\}$ and $c_\EC=1$
\item $\EC$ is of type $T \in \{ \mathbf{I, II, III, IV, V, VI, VII, XI, XIII, XV}\}$ with $c_\EC >1$ and either $q \equiv 1 \pmod 8$, or $q$ is a prime of supersingular reduction for $\EC$ with $q \equiv 3 \pmod 4$
\item $\EC$ is of any Ivorra type, $c_\EC=1$, and either $q$ is a prime of supersingular reduction for $\EC$, or $q \equiv 1 \pmod 4$
\end{enumerate}
\end{theorem}

Table~\ref{table 4} summarizes\footnote{The family of elliptic curves $\{\mathbf{I, II, III, IV, V, VI, VII, XI, XIII, XV}\}$ is together referred to as the \emph{trick family}.}  the cases where Watkin's Conjecture holds for the quadratic twist $\EC^{(\pm q)}/\Q$.

\begin{table}[h!]
    \centering
\caption{Cases where Watkin's Conjecture holds for the quadratic twist $\EC^{(\pm q)}$.
}
\label{table 4}

    \begin{tabular}{| c || c | p{2.5cm} | c |}
         \hline
         \emph{family} & \emph{Condition on} $c_{\EC}$ & \emph{Condition on } $q$ & \emph{Proof} \\
         \hline   
         \hline
         \emph{trick} & $c_{\EC}=1$ & - & \emph{Lemma~\ref{Lemma 4.3}}\\
         \hline
         \emph{trick} & $c_{\EC}>1$ & $q\equiv 1\pmod{8}$ & \emph{Lemma~\ref{Lemma 4.4}}\\
         \hline
         \emph{trick} & $c_{\EC}>1$ & $q\equiv 3\pmod{4}$; \newline
         \emph{supersingular} & \emph{Lemma~\ref{Lemma 4.4}}\\
         \hline \hline
         \emph{all} & $c_{\EC}=1$ & \emph{supersingular} & \emph{Lemma~\ref{Lemma 4.5}}\\
         \hline
         \emph{all} & $c_{\EC}=1$ & $q\equiv 1\pmod{4}$ & \emph{Lemma~\ref{Lemma 4.5}}\\
         \hline
    \end{tabular}
\end{table}

\begin{remark}
Cases which remain unaddressed by our previous theorem are the following:
\begin{enumerate}
\item $\EC$ is a rank 0 Ivorra curve in the trick family with $c_{\EC} >1$ and
\begin{enumerate}
\item $q \equiv 5\pmod{8}$ \emph{or}
\item $q \equiv 3\pmod{4}$ and $a_q(\EC)\neq 0$.
\end{enumerate}
\item $\EC$ is an Ivorra curve \textbf{not} in the trick family and
\begin{enumerate}
    \item $c_{\EC} =1$, $q \equiv 3\pmod{4}$, and $a_q(\EC)\neq 0$ \emph{or}
    \item $c_{\EC} >1$ and $\EC$ has rank 0.
\end{enumerate}
\end{enumerate}
Regarding case (1), our computational experiments have only ever found Ivorra curves with $c_\EC>1$ when the curve is of type $\textbf{I}$, and furthermore, each such curve has had rank 0, as expected (see Remark~\ref{rmk:manin-constant-1}).

One possible strategy for handling some of these missing cases would be to use the results of \cite{KrizLi}, but their Hypothesis ($\star$) appears not to hold for Ivorra curves.
Even though several authors have studied the question of ranks in quadratic twist families, it appears that the behavior of quadratic twists by primes $q \equiv 3 \pmod 4$ is especially challenging, in general.
For instance, the results of \cite{CaiLiZhai} work only for twists by primes $q \equiv 1 \pmod 4$, and the hypotheses of \cite[Theorem~1.5]{Zhai}, while widely applicable, seem never to be satisfied by Ivorra curves.
\end{remark}

\appendix
\section{Missing Cases of (Weaker) Watkins' Conjecture}\label{appendix:clarify-caro}

A weaker version of Watkins' conjecture asserts that 
\begin{equation}
\label{A.1}
 2 \nmid m_{\EC} \Longrightarrow \rk_\Z(\EC(\Q))= 0. 
\end{equation}
In the introduction of \cite{caro2022watkins}, Caro claims that this conjecture is largely settled, with a few restrictive cases remaining. We now clarify this claim. 

\begin{proposition}
Let $\EC / \Q$ be an elliptic curve of conductor $N$ 
and odd modular degree $m_\EC$.
Then 
the missing cases of \eqref{A.1} are when \emph{all} the following conditions are satisfied 
\begin{itemize}
\item $N$ is divisible by at most two odd primes,  
\item $\EC$ has additive reduction at $2$, and
\item $\EC$ has a rational point of order $2$.
\end{itemize}

\begin{proof}
First, since we are assuming that $2 \nmid m_{\EC}$, it follows from \cite[Theorem~1.1(1)]{CE09} that $N$ is divisible by at most two odd primes. 

Next, if $4 \nmid N$, then \cite[Theorem~2.1]{agashe2012modular} asserts that $m_\EC$ is equal to the congruence number $\delta_\EC$ of $\EC$.
Thus $\delta_\EC$ is odd, but then the theorem of \cite{KKwatkins} implies that $\mathrm{rank}_\Z(\EC(\Q))=0$.
So the missing case is when $4 \mid N$, hence $\EC$ has additive reduction at $2$.

These deductions put us in case (a) of \cite[Theorem~1.1(3)]{CE09}, hence $\EC(\Q)[2]$ is non-trivial.
\end{proof}

\end{proposition}

\section{Complete $2$-descent proof for Theorem~\ref{thm:rank-bound-for-ivorra-curves}}
\label{appendix}
We now give the full proof of Theorem~\ref{thm:rank-bound-for-ivorra-curves}.

\begin{proof}
Consider the curves
\begin{align*}
{\EC} \colon & y^2 =x^3 + ax^2 + bx \\
{\EC}' \colon & y^2 = x^3 - 2ax^2 + (a^2-4b)x.
\end{align*}
Both ${\EC}(\Q)[2]$ and ${\EC}'(\Q)[2]$ consist of $K=\{ \mathcal{O}, (0,0) \}$, and there is a $2$-isogeny $\phi \colon {\EC} \to {\EC}'$ with kernel $K$.
Let $\phi'$ denote the dual isogeny.
Since ${\EC}$ and ${\EC}'$ are $2$-isogenous, we have 
\[
\mathrm{rank}_\Z {\EC}(\Q)=\mathrm{rank}_\Z {\EC}'(\Q).
\]

To bound these ranks, we use $2$-descent.
By \cite[Proposition~1.1]{aguirre2008elliptic}, case (1) is immediate, so we focus on the remaining cases and carry out the $2$-descent more explicitly, following the method outlined in \cite[Chapter~X]{silverman}.

Each of these curves has conductor $2^m p$ for some integer $m \geq 2$.
Let $\Sigma =\{\pm 1, \pm 2, \pm p, \pm 2p\}$.
To each $d \in \Sigma$ we have the associated homogeneous spaces
\begin{align*}
C_d \colon & dw^2 = d^2-2 a d z^2 + (a^2-4b) z^4 \\
C_d' \colon  & dw^2 = d^2+4 a d z^2 -16b z^4.
\end{align*}
For $\phi$ (and analogously for $\phi'$), we obtain Selmer groups
\[
\Sel^{(\phi)}({\EC}/\Q) = \{ d\in \Sigma \ | \ C_d(\Q_v) \neq \emptyset \text{ for every place } v \in \Sigma \}.
\]
In particular, $\Sel^{(\phi)}({\EC}/\Q)$ is a $2$-group.
There are injections 
\begin{equation*}
\label{eq:selmer-injection}
{\EC}'(\Q) /  \phi({\EC}(\Q)) \xhookrightarrow{\delta} \Sel^{(\phi)}({\EC}/\Q) \quad \text{and} \quad  {\EC}(\Q) /  \phi'({\EC}'(\Q)) \xhookrightarrow{\delta'} \Sel^{(\phi')}({\EC}'/\Q)
\end{equation*}
with the explicit values
\[
\delta(\mathcal{O}) = 1 \quad \text{and} \quad
\delta((0,0)) = \text{the square-free part of } a^2-4b,
\]
and similarly for $\delta'$.
We record this information in Table~\ref{table: 2-descent}, along with everything else we deduce in the rest of the descent argument.

By \cite[Equation~(5)]{aguirre2008elliptic}, we have
\begin{equation}
\label{eq:selmer-rank-bound}
r \leq \dim_{\mathbb{F}_2} \Sel^{(\phi)}({\EC}/\Q) + \dim_{\mathbb{F}_2} \Sel^{(\phi')}({\EC}'/\Q) - 2.
\end{equation}

We now compute these Selmer groups by studying the homogeneous spaces defined above.
We will not give details for the family $T=\mathbf{X}$ again as it was already included in the main text.

\begin{center}
\begin{table}[h!]
\caption{This table summarizes the findings from performing $2$-descent on each family of elliptic curves.
For each type, the first row gives information about $\Sel^{(\phi)}({\EC}/\Q)$, and the second row gives information about $\Sel^{(\phi')}({\EC'}/\Q)$ A green cell indicates a global point in the corresponding Selmer group coming from $2$-torsion.
A red cell indicates that the corresponding homogeneous space has no local solutions over the indicated field, or that using the group structure of the Selmer group it can be deduced that it does not contain this element.
Finally, a blue cell indicates that we can omit analyzing this cell to obtain our bound on the Mordell--Weil rank.}
\begin{tabular}{|c|c|c|c|c|c|c|c|c|}
\hline
T                               & \textbf{$1$}                                               & \textbf{$-1$}                                       & \textbf{$2$}                                        & \textbf{$-2$}                                       & \textbf{$p$}                                       & \textbf{$-p$}                                      & \textbf{$2p$}                                       & \textbf{$-2p$}                            \\ \hline \hline

\multirow{2}{*}{\textbf{VIII}, $k=2$}                     & \cellcolor{green!25}$\delta(\mathcal{O})$ & \cellcolor{red!25}$\R$  &  \cellcolor{red!25}$\Q_2$    &  \cellcolor{red!25}$\Q_2$       &  \cellcolor{green!25}$\delta(0,0)$       & \cellcolor{red!25}grp     &      \cellcolor{red!25}grp     & \cellcolor{red!25}grp \\ \cline{2-9} 
                                        & \cellcolor{green!25}$\delta(\mathcal{O})$ & \cellcolor{green!25}$\delta(0,0)$             & \cellcolor{red!25}$\Q_2$ & \cellcolor{red!25}grp             &  \cellcolor{blue!25}                   &  \cellcolor{blue!25}        & \cellcolor{red!25}$\Q_p$
                                        &  \cellcolor{red!25}grp         \\ \hline \hline
                       
\multirow{2}{*}{\textbf{X}, $k$ even}                         & \cellcolor{green!25}$\delta(\mathcal{O})$   &          \cellcolor{blue!25}           &               \cellcolor{red!25}$\Q_2$  &\cellcolor{red!25}$\Q_2$     &         \cellcolor{blue!25}           & \cellcolor{green!25}$\delta(0,0)$ &                              \cellcolor{red!25}grp          &   \cellcolor{red!25}grp            \\ \cline{2-9} 
                                          & \cellcolor{green!25}$\delta(\mathcal{O})$ & \cellcolor{red!25}$\R$             &  \cellcolor{blue!25}
                                          & 
                                          \cellcolor{red!25}$\R$             & \cellcolor{red!25}$\Q_p$          & \cellcolor{red!25}$\Q_p$          & \cellcolor{red!25}$\Q_p$           & \cellcolor{red!25}$\Q_p$ \\ \hline \hline

\multirow{2}{*}{\textbf{X}, $k$ odd}                         & \cellcolor{green!25}$\delta(\mathcal{O})$ &       \cellcolor{blue!25}              &    \cellcolor{red!25}$\Q_2$        &  \cellcolor{red!25}$\Q_2$        &          \cellcolor{blue!25}          & \cellcolor{green!25}$\delta(0,0)$ & \cellcolor{red!25}grp  &   \cellcolor{red!25}grp            \\ \cline{2-9} 
                                          & \cellcolor{green!25}$\delta(\mathcal{O})$ & \cellcolor{red!25}$\R$             & \cellcolor{green!25}$\delta(0,0)$ & \cellcolor{red!25}$\R$             & \cellcolor{red!25}$\Q_p$          & \cellcolor{red!25}$\Q_p$          & \cellcolor{red!25}$\Q_p$           & \cellcolor{red!25}$\Q_p$ \\ \hline \hline

\multirow{2}{*}{\textbf{XII}, $k=2$}                     & \cellcolor{green!25}$\delta(\mathcal{O})$ & \cellcolor{green!25}$\delta(0,0)$  &  \cellcolor{blue!25}                                              &              \cellcolor{blue!25}                                   &  \cellcolor{red!25}$\Q_p$      & \cellcolor{red!25}$\Q_p$      &      \cellcolor{red!25}$\Q_p$      & \cellcolor{red!25}$\Q_p$ \\ \cline{2-9} 
                                        & \cellcolor{green!25}$\delta(\mathcal{O})$ & \cellcolor{red!25}$\R$             & \cellcolor{green!25}$\delta(0,0)$ & \cellcolor{red!25}$\R$             &  \cellcolor{red!25}$\Q_p$                   &  \cellcolor{red!25}$\Q_p$         & \cellcolor{red!25}grp
                                        &  \cellcolor{red!25}grp         \\ \hline \hline

\multirow{2}{*}{\textbf{XIV}}                     & \cellcolor{green!25}$\delta(\mathcal{O})$ &          \cellcolor{blue!25}           &          \cellcolor{red!25}$\Q_2$           &  \cellcolor{red!25}$\Q_2$   &     \cellcolor{blue!25}      &  \cellcolor{blue!25}      &    \cellcolor{red!25}$\Q_2$    &    \cellcolor{red!25}$\Q_2$   \\ \cline{2-9} 
                                       & \cellcolor{green!25}$\delta(\mathcal{O})$ &\cellcolor{red!25}$\Q_2$                     & \cellcolor{green!25}$\delta(0,0)$  &       \cellcolor{red!25}grp              & \cellcolor{red!25}$\Q_p$          & \cellcolor{red!25}$\Q_p$          & \cellcolor{red!25}$\Q_p$           & \cellcolor{red!25}$\Q_p$ \\ \hline \hline

\multirow{2}{*}{\textbf{XVI}}                     & \cellcolor{green!25}$\delta(\mathcal{O})$ & \cellcolor{red!25}$\R$  &  \cellcolor{red!25}$\Q_2$    &  \cellcolor{red!25}$\Q_2$       &  \cellcolor{green!25}$\delta(0,0)$       & \cellcolor{red!25}grp     &      \cellcolor{red!25}grp     & \cellcolor{red!25}grp \\ \cline{2-9} 
                                        & \cellcolor{green!25}$\delta(\mathcal{O})$ & \cellcolor{blue!25}             & \cellcolor{blue!25} & \cellcolor{green!25}$\delta(0,0)$             &  \cellcolor{red!25}$\Q_p$                   &  \cellcolor{red!25}grp        & \cellcolor{red!25}grp
                                        &  \cellcolor{red!25}grp         \\ \hline

\multirow{2}{*}{\textbf{XVII}, $k=1$}                  & \cellcolor{green!25}$\delta(\mathcal{O})$ & \cellcolor{red!25}$\R$             &\cellcolor{red!25}grp     &    \cellcolor{red!25}$\R$       &    \cellcolor{red!25}$\Q_p$         &      \cellcolor{red!25}$\Q_p$      & \cellcolor{green!25}$\delta(0,0)$ & \cellcolor{red!25}grp \\ \cline{2-9} 
                                     & \cellcolor{green!25}$\delta(\mathcal{O})$ &   \cellcolor{red!25}grp                  &    \cellcolor{red!25}$\Q_2$                 & \cellcolor{green!25}$\delta(0,0)$  & \cellcolor{red!25}$\Q_p$          & \cellcolor{red!25}$\Q_p$          & \cellcolor{red!25}$\Q_p$           & \cellcolor{red!25}$\Q_p$ \\ \hline \hline

\multirow{2}{*}{\textbf{XVIII}, $k=1$}                 & \cellcolor{green!25}$\delta(\mathcal{O})$ & \cellcolor{red!25}$\Q_2$           & \cellcolor{red!25}grp              & \cellcolor{green!25}$\delta(0,0)$  &      \cellcolor{red!25}$\Q_p$   &   \cellcolor{red!25}$\Q_p$       &     \cellcolor{red!25}$\Q_p$          & \cellcolor{red!25}$\Q_p$   \\ \cline{2-9} 
                                      & \cellcolor{green!25}$\delta(\mathcal{O})$ & \cellcolor{red!25}$\R$             &     \cellcolor{blue!25}                & \cellcolor{red!25}$\R$             &   \cellcolor{blue!25}                 & \cellcolor{red!25}grp             & \cellcolor{green!25}$\delta(0,0)$  & \cellcolor{red!25}grp    \\ \hline \hline

\multirow{2}{*}{\textbf{XIX}, $k=1$}                & \cellcolor{green!25}$\delta(\mathcal{O})$ & \cellcolor{red!25}$\Q_2$           &   \cellcolor{red!25}grp  &   \cellcolor{red!25}grp       &  \cellcolor{red!25}$\Q_2$        &  \cellcolor{red!25}$\Q_2$     & \cellcolor{green!25}$\delta(0,0)$ & \cellcolor{red!25}grp \\ \cline{2-9} 
                                     & \cellcolor{green!25}$\delta(\mathcal{O})$ &   \cellcolor{blue!25}        & \cellcolor{green!25}$\delta(0,0)$  &       \cellcolor{blue!25}       & \cellcolor{red!25}$\Q_p$          & \cellcolor{red!25}$\Q_p$          & \cellcolor{red!25}$\Q_p$           & \cellcolor{red!25}$\Q_p$ \\ \hline \hline

\multirow{2}{*}{\textbf{XIX}, $k=2$}                & \cellcolor{green!25}$\delta(\mathcal{O})$ & \cellcolor{red!25}$\Q_2$           & \cellcolor{green!25}$\delta(0,0)$ & \cellcolor{red!25}grp           &    \cellcolor{red!25}$\Q_2$      & \cellcolor{red!25}$\Q_2$     &\cellcolor{red!25}grp   &\cellcolor{red!25}grp  \\ \cline{2-9} 
                                     & \cellcolor{green!25}$\delta(\mathcal{O})$ &   \cellcolor{blue!25}        & \cellcolor{green!25}$\delta(0,0)$  &      \cellcolor{blue!25}        & \cellcolor{red!25}$\Q_p$          & \cellcolor{red!25}$\Q_p$          & \cellcolor{red!25}$\Q_p$           & \cellcolor{red!25}$\Q_p$ \\ \hline

\end{tabular}
\end{table}
\label{table: 2-descent}
\end{center}

\noindent\textbf{(i)} \fbox{$d=-1 :$} Consider the homogeneous spaces 
\begin{align*}
C_{-1} \colon & -w^2 = 1 + 2 a  z^2 + (a^2-4b) z^4 \\
C'_{-1} \colon & -w^2 = 1- 4 a z^2 -16b z^4.
\end{align*}

\smallskip

\noindent \textbf{$\R$-obstruction:} For both spaces the left-hand side is always non-positive, while the right-hand side takes a positive value when $z=0$.
Viewing the right-hand side as a quadratic in $z^2$, the discriminant is \textit{negative} for $C_{-1}$ precisely when $b<0$, i.e., when $T \in \{\mathbf{VII,XVI, XVII}\}$, and the discriminant is \textit{negative} for $C'_{-1}$ precisely when $a^2-4b <0$, i.e., when $T \in \{ \mathbf{X, XII, XVIII}\}$.
In these cases, the homogeneous space has no $\R$-solutions; and $-1$ is not in the relevant Selmer group.

\smallskip
  
\noindent \textbf{$\Q_2$-obstruction:}

$\bullet$
Suppose that $T \in \{\mathbf{XVIII, XIX}\}$ and suppose we have a solution $(w,z) \in C_{-1}(\Q_2)$.
In these families we have $\ord_2(a)= \ord_{2}(\pm 4\alpha) = 2 + \epsilon$ for some nonnegative integer $\epsilon$, and $\ord_2(a^2-4b)=3$.
On the left-hand side of the equation, we have that $\ord_2(-w^2)$ is even.
Letting $j=\ord_2(z)$, the right-hand side gives 
\[
\ord_2(\textrm{RHS}) = \ord_2(1 + 2 a  z^2 + (a^2-4b) z^4 ) \geq \min\{0, (3+\epsilon)+2j,3+4j\}.
\]
Since $3+4j$ is odd, we must have $w, z \in \Z_2$.
Reducing mod $4$ then implies $1 \equiv -w^2 \pmod 4$ which is impossible, so $C_{-1}(\Q_2)=\emptyset$ for these types $T$.

$\bullet$ Consider now $T=\mathbf{XIV}$.
We have the homogeneous space 
\[
C'_{- 1} \colon -w^2 = 1 \mp 8 \alpha z^2 - 32z^4.
\]
Suppose $w,z \in \Q_2$ give a solution; then comparing $2$-adic valuations, we see that $w \in \Z_2^\times$.
If $z \in \Z_2$, then reducing mod 4 would give $-1 \equiv 1 \pmod 4$ which is false, so the only other possibility is that $z=\frac{1}{2}\zeta$ with $\zeta \in \Z_2^\times$.
Substituting and simplifying, we obtain 
\[
-w^2 = 1 \mp 2 \alpha \zeta^2 - 2 \zeta^4.
\]
Since 1 is the only odd square mod $4$, this implies
\[
\pm 2 \alpha \equiv 0 \pmod 4,
\]
but $\alpha$ is odd so this is impossible.
Thus $C'_{-1}(\Q_2)=\emptyset$.

\smallskip

\noindent \textbf{(ii)} \fbox{$d=\pm 2$:}
First, let us consider the homogeneous spaces
\[
C_{\pm2} \colon  \pm2w^2 = 4 \mp4 a z^2 + (a^2-4b) z^4.
\]

\smallskip

\noindent \textbf{$\R$-obstruction:} 
We begin by focusing on the case $d=-2$.
In this case, the left-hand side is always non-positive, while the right-hand side certainly takes positive values.
Viewed as a quadratic in $z^2$, the discriminant on the right-hand side is $64b$, so it is negative precisely when $b<0$.
In particular, this shows that $C'_{-2}(\R)=\emptyset$ when $T=\mathbf{XVII}$.
(While it also applies to $T \in \{\mathbf{VIII,XVI}\}$, below we give an argument which handles both cases $d=\pm 2$ uniformly.)

\smallskip

\noindent \textbf{$\Q_2$-obstruction:}

$\bullet$ For $T=\mathbf{VIII}$ with $k=2$, these homogeneous spaces specialize to  
\[
C_{\pm 2} \colon \pm 2 w^2 = 4 \mp 8 \alpha z^2 + 4p z^4
\]
where there is some ambiguity about the sign on the middle term of the RHS, but this ambiguity does not affect our argument.
We have $\ord_2(\mathrm{LHS})$ is odd, while
\[
\ord_2(\mathrm{RHS}) \geq \min \{2, 3+2j, 2+4j\}
\]
where $j=\ord_2(z)$.
Looking mod 4 shows that $w \in 2 \Z_2$, so $\ord_2(\mathrm{LHS})\geq 3$.
The only way to achieve this is if $\ord_2(\mathrm{LHS})=3$ and $j=0$.
Write $w=2W$.
Substituting gives
\[
8W^2 = 4 \mp 8 \alpha z^2 + 4pz^4
\]
and simplifying yields 
\[
2W^2 = 1 \mp 2 \alpha z^2 + pz^4.
\]
Looking mod $8$ and noting that $p \equiv 5 \pmod{8}$, we have 
\begin{align*}
2 & \equiv 1 \mp 2 \alpha + 5 \pmod 8 \textrm{ or } \\
4 & \equiv 2 \alpha \pmod 8.
\end{align*}
But $\alpha$ is an odd integer, so this is impossible, and this $\Q_2$-obstruction shows that $\pm 2 \not\in \Sel^{(\phi)}(\EC/\Q)$.
Since $p \in \Sel^{(\phi)}(\EC/\Q)$, the group structure now allows us to deduce that $\Sel^{(\phi)}(\EC/\Q) \simeq \{1, p \}$.

$\bullet$ If we consider the same homogeneous space with $T=\mathbf{XVI}$, we have 
\[
C_{\pm 2} \colon \pm 2w^2 = 4 \mp 8 \alpha z^2 + 4pz^4
\]
where there is some ambiguity on the middle sign on the RHS, but it does not affect the argument.
If we have a solution with $w,z \in \Q_2$, then $\ord_2(2w^2)$ is odd, but
\[
\ord_2(\mathrm{RHS})\geq \min \{2, 3+2j, 4+4j\}
\]
with equality unless at least two of the terms are equal.
However, for every value of $j$ these terms are distinct, and the minimum is even.
Thus, we have a $\Q_2$-obstruction, hence $\pm 2 \not\in \Sel^{(\phi)}(\EC/\Q)$.

The exact same argument (\textit{mutatis mutandis}) also shows that $C_{\pm 2p}(\Q_2)=\emptyset$ when $T=\mathbf{XVI}$.

\smallskip

Now let us consider the homogeneous space
\[
C'_{-2} \colon  - 2w^2 = 4 - 8 a z^2 -16b z^4.
\]

\smallskip 

\noindent \textbf{$\R$-obstruction:}
The left-hand side is always non-positive, while the right-hand side certainly takes positive values.
Viewed as a quadratic in $z^2$, the discriminant on the right-hand side is $64(a^2-4b)$, so it is negative precisely when $a^2-4b<0$.
This shows that $C'_{-2}(\R)=\emptyset$ when $T\in\{ \mathbf{X, XII, XVIII} \}$, so $-2 \notin \Sel^{(\phi')}(\EC'/\Q)$ for these $T$.

\smallskip 

Finally, we consider the homogeneous space $C'_2$ for two families.

$\bullet$ For $T=\mathbf{VIII}$ with $k=2$, we have
\[
C'_{2} \colon 2w^2 = 4 \pm 16 \alpha z^2 + 2^6 z^4.
\]
Once again $\ord_2(\mathrm{LHS})$ is odd, while
\[
\ord_2(\mathrm{RHS}) \geq \min \{2, 4+2j, 6+4j\}
\]
with $j=\ord_2(z)$.
The only way for this to work is if $\ord_2(w)=1$ and $j=-1$, so we write $w=2W$ and $z=\frac{1}{2}\zeta$ with $W,\zeta \in \Z_2^\times$.
Substituting gives 
\[
8W^2 = 4 \pm 4 \alpha \zeta^2 + 4 \zeta^4,
\]
and simplifying yields
\[
2W^2 = 1 \pm \alpha \zeta^2 + \zeta^4.
\]
Reducing this mod 8, we have
\[
2 \equiv 2 \pm \alpha \pmod 8,
\]
but this is impossible since $\alpha$ is an odd integer.
So $2 \not\in \Sel^{(\phi')}(\EC'/\Q)$, and the group structure allows us to deduce $-2 \not\in \Sel^{(\phi')}(\EC'/\Q)$ as well.

A similar argument (except the punchline is that $0 \equiv 5 \alpha \pmod 8$) shows that $2p \not\in \Sel^{(\phi')}(\EC'/\Q)$, and the group structure implies also $-2p \not\in \Sel^{(\phi')}(\EC'/\Q)$.

$\bullet$ Now consider the $C'_{2}$ for $T=\mathbf{XVII}$, which is
\[
2w^2 = 4 \pm 32 \alpha z^2 - 32 z^4,
\]
where $\alpha$ is odd by assumption (since $p \equiv 3 \pmod 8$ and $\alpha^2=(p-1)/2$).

\smallskip

\noindent \textbf{$\Q_2$-obstruction:} We have $\ord_2(2w^2)$ is odd, while writing $j=\ord_2(z)$, we have
\[
\ord_2(\mathrm{RHS}) \geq \min \{2, 5+2j, 5+4j\}.
\]
Since $5+4j$ is odd, this implies $w \in \Z_2$.
If $z \in \Z_2$, reducing this equation mod 4 implies $\ord_2(w)>0$, and then reducing mod 8 gives the contradiction $0 \equiv 4 \pmod 8$.
So instead we must have $j<0$, and since $w \in \Z_2$, the only possibility is $j=-1$.
Write $z=\frac{1}{2}\zeta$ with $\zeta \in \Z_2^\times$.
Substituting and simplifying gives
 \[
 w^2 = 2 \pm 4\alpha \zeta^2 - z^4,
 \]
 and reducing this mod 8 gives
 \[
 4\alpha \equiv 0 \pmod 8.
 \]
 But since our assumptions for family $T=\mathbf{XVII}$ imply that $\alpha$ is odd, we have a contradiction, hence we have a $\Q_2$-obstruction in this case.

\smallskip

\noindent \textbf{(iii)} \fbox{$d=\pm p$:}
Consider the homogeneous spaces
\[
C'_{\pm p} \colon \pm p w^2 = p^2 \pm 4p a z^2 -16b z^4.
\]

\noindent \textbf{$\Q_p$-obstruction:}

$\bullet$ First suppose that $T \in \{\mathbf{X, XIV, XVI, XVII, XIX }\}$.
We have $p \nmid ab$.
Suppose there is a solution $(w,z) \in C'_{\pm p}(\Q_p)$.
Then $\ord_p(pw^2)$ is odd and 
\[
\ord_p(\mathrm{RHS}) \geq \min \{ 2, 1+2j, 4j \}
\]
with $j = \ord_p(z)$.
Since $4j$ is even, we must have $w, z \in \Z_p$.
Reducing mod $p$ shows that $z \in p\Zp$, and then reducing mod $p^2$ shows that $w \in p\Zp$, but then this implies $p^2 \equiv 0 \pmod p^3$, a contradiction.

For the same set of $T$, if we consider the form of the homogeneous spaces $C'_{\pm 2p}$, the same argument applies again, and we deduce that $\pm p, 2p \not\in \Sel^{(\phi')}(\EC'/\Q)$ for $T \in \{\mathbf{X, XIV, XVI, XVII, XIX }\}$.

$\bullet$ Now let us consider $T=\mathbf{XII}$ with the additional hypotheses that $k=2$ and $-64$ is not a fourth power mod $p$.
We have the homogeneous space 
\[
C'_{\pm p} \colon \pm p w^2 = p^2 \pm 8p \alpha z^2 - 32 p^2 z^4.
\]
Suppose there is a solution with $w,z \in \Q_p$, and write $j=\ord_p(z)$.
Then $\ord_p(pw^2)$ is odd, while 
\[
\ord_p(\mathrm{RHS}) \geq \min \{ 2, 1+2j, 2+4j \},
\]
from which we deduce that $w, z \in \Z_p$.
Reducing mod $p^2$ gives
\[
w^2 \equiv 8 \alpha z^2 \pmod p,
\]
which implies
\begin{align*}
w^4 &\equiv 64 z^4 (2p^2-1) \pmod p \textrm{ or }\\
w^4 &\equiv -64 z^4 \pmod p,
\end{align*}
which is impossible by assumption.
So $C'_{\pm p}(\Q_p)=\emptyset$.

\smallskip

Now consider the homogeneous spaces
\[
C_{\pm p} \colon \pm pw^2 = p^2 \mp 2paz^2 + (a^2-4b)z^4.
\]

\noindent \textbf{$\R$-obstruction:}

$\bullet$
When $d=-p$, the left-hand side is always non-positive, whereas the right-hand side certainly takes positive values.
Viewing the right-hand side as a quadratic in $z^2$, its discriminant is
\[
4p^2a^2-4p^2(a^2-4b)=16b,
\]
so we have $C_{-p}(\R)=\emptyset$ when $b<0$.
This shows that $-p \notin \Sel^{(\phi)}(\EC/\Q)$ for $T \in \{\mathbf{XVI, XVII}\}$.

\smallskip

\noindent \textbf{$\Q_2$-obstruction:}

$\bullet$
Consider $T \in \{\mathbf{XVII, XIX}\}$ and suppose there is a solution $(w,z) \in C_{\pm p}(\Q_2)$.
Then $\ord_2(pw^2)$ is even, and writing $\ord_p(z)=j$ we see
\[ 
\ord_2(p^2 \mp 2paz^2 + (a^2-4b)z^4) \geq \min\{0, 3+\epsilon+2j, 3+4j\}
\]
for some nonnegative integer $\epsilon$.
Since $3+4j$ is odd, we must have $w, z \in \Z_2$.
Reducing mod $8$ then implies $w^2=p \pmod 8$, but $1$ is the only odd square mod $8$, and by assumption $p \not\equiv 1 \pmod 8$.
Thus $C_{\pm p}(\Q_2)=\emptyset$ for $T \in \{\mathbf{XVII, XVIII, XIX} \}$ under the additional hypotheses of the theorem.

\smallskip

\noindent \textbf{$\Q_p$-obstruction:}

$\bullet$
For $T\in\{ \mathbf{XII, XVIII}\}$, the homogeneous spaces are
\[
C_{\pm p} \colon \pm pw^2 = p^2 \mp 4p\alpha z^2 -4 z^4.
\]
Suppose $(w,z) \in C_{\pm p} (\Q_p)$ and let $j=\ord_p(z)$.
Then $\ord_p(pw^2)$ is odd, while
\[
\ord_p(p^2 \mp 4p \alpha z^2 - 4z^4) \geq \min\{ 2, 1 +2j, 4j \},
\]
so we must have $w, z \in \Z_p$.
Reducing mod $p$ shows $z \in p\Z_p$, after which reducing mod $p^2$ shows $w \in p\Z_p$, but then we arrive at the contradiction $p^2 \equiv 0 \pmod p^3$, so $C_{\pm p} (\Q_p)= \emptyset$.

The same argument shows $C_{\pm 2p} (\Q_p)=\emptyset$ in this case.
We conclude that $\pm p, \pm 2p \notin \Sel^{(\phi)}(\EC/\Q)$ for $T\in\{\mathbf{XII,XVIII}\}$.

\smallskip

\noindent \textbf{(iv)} \fbox{Group structure obstructions:}
Recall that our Selmer groups have been identified with a subgroup of $\Q^\times / (\Q^\times)^2$.
It is now possible to use this group structure to eliminate a few more cases, just as in the proof for $T=\mathrm{X}$.
Carrying on in this manner, we complete Table~\ref{table: 2-descent}.

There remain some undetermined cells, but nevertheless, we see that in every case we have 
\[
\dim_{\mathbb{F}_2} \Sel^{(\phi)}(\EC/\Q) + \dim_{\mathbb{F}_2} \Sel^{(\phi')}(\EC'/\Q) \leq  3.
\]
So ~\eqref{eq:selmer-rank-bound} implies that
\[
\mathrm{rank}_\Z \EC(\Q), \mathrm{rank}_\Z \EC'(\Q) \leq 1.
\qedhere
\]
\end{proof}

\section{Dictionary to go between Ivorra's paper and our classification types}
\label{appendix-dictionary}

\begin{table}[h!]
\begin{tabular}{|c||c|}
\hline
{\textbf{Type}} & {\textbf{Ivorra Theorem \# and Curve Labels}} \\ \hline
\hline

\textbf{I}          & 2A, 3A, 3B, 4D, 5B, 5B'                           \\ \hline
\textbf{II}         & 3C, 5C, 5C'                                       \\ \hline
\textbf{III}        & 4A                                                \\ \hline
\textbf{IV}         & 4B                                                \\ \hline
\textbf{V}          & 4E                                                \\ \hline
\textbf{VI}         & 5A, 5A'                                           \\ \hline
\textbf{VII}        & 6A, 6A'                                           \\ \hline
\textbf{VIII}       & 6B, 6B'                                           \\ \hline
\textbf{IX}         & 6C, 6C'                                           \\ \hline
\textbf{X}          & 6E, 6E'                                           \\ \hline
\textbf{XI}         & 7A, 7A'                                           \\ \hline
\textbf{XII}        & 7B, 7B'                                           \\ \hline
\textbf{XIII}       & 7C, 7C'                                           \\ \hline
\textbf{XIV}        & 7D, 7D'                                           \\ \hline
\textbf{XV}         & 7E, 7E'                                           \\ \hline
\textbf{XVI}        & 7F, 7F'                                           \\ \hline
\textbf{XVII}       & 8A, 8A'                                           \\ \hline
\textbf{XVIII}      & 8B, 8B'                                           \\ \hline
\textbf{XIX}        & 8C, 8C'                                           \\ \hline
\textbf{XX}         & 8D, 8D'                                           \\ \hline
\end{tabular}
\caption{A dictionary for translating between our classification types and the curves listed in Ivorra \cite{ivorra2004courbes}.}
\label{table:dictionary}
\end{table}

\newpage

\bibliographystyle{amsalpha}
\bibliography{references}

\end{document}